\documentclass[12pt]{amsart}
\newcommand{\hide}[1]{} \textwidth16cm \hoffset=-2truecm
\usepackage{amsmath}
\usepackage{amsthm}
\usepackage{lipsum}
\usepackage{amsopn}
\usepackage{hyperref}
\usepackage{enumerate}
\numberwithin{equation}{section}
\usepackage{float}
\usepackage[usenames,dvipsnames]{xcolor}
\usepackage{tcolorbox}
\usepackage{graphicx}
\usepackage{titlesec}
\usepackage{mathptmx} 
\usepackage{fancyhdr}
\usepackage{xcolor} 
\usepackage{textcomp}
\usepackage{fancyhdr}
\newcommand\smallO{
  \mathchoice
    {{\scriptstyle\mathcal{O}}}
    {{\scriptstyle\mathcal{O}}}
    {{\scriptscriptstyle\mathcal{O}}}
    {\scalebox{.7}{$\scriptscriptstyle\mathcal{O}$}}
  }

\newtheorem{corollary}{Corollary}[section]
\newtheorem{lemma}{Lemma}[section]
\newtheorem{theorem}{Theorem}[section]
\newtheorem{proposition}{Proposition}[section]

\theoremstyle{definition}
\newtheorem{remark}{Remark}[section]

\DeclareMathOperator{\De}{\Delta}

\DeclareMathOperator{\ArE}{A_{\textrm{e}}}
\DeclareMathOperator{\ArS}{A_{\textrm{s}}}

\DeclareMathOperator{\Le}{L}
\DeclareMathOperator{\LeE}{L_{\textrm{e}}}
\DeclareMathOperator{\LeS}{L_{\textrm{s}}}

\DeclareMathOperator{\AreS}{\mathcal{A}_{\textrm{s}}}

\DeclareMathOperator{\LenS}{\mathcal{L}_{\textrm{s}}}

\DeclareMathOperator{\T}{\mathbb{T}}

\DeclareMathOperator{\Rot}{Rot}
\DeclareMathOperator{\D}{\mathbb{D}}
\DeclareMathOperator{\RE}{Re}
\DeclareMathOperator{\CC}{\widehat{\mathbb{C}}}
\DeclareMathOperator{\IM}{Im}
\titleformat{\section}
  {\normalfont\scshape}{\thesection}{1em}{\centering}

\begin{document}
\title[Geometric versions of Schwarz's lemma for spherically convex functions]{GEOMETRIC VERSIONS OF SCHWARZ'S LEMMA FOR\\[2mm] SPHERICALLY CONVEX FUNCTIONS}

\author[Maria Kourou]{Maria Kourou$^{\S\ }$}
\thanks{$^{\S}$Partially supported by the Alexander von Humboldt Foundation.}
\address{Department of Mathematics, University of W\"urzburg, 97074, W\"urzburg, Germany}
\email{ maria.kourou@mathematik.uni-wuerzburg.de}   

\author[Oliver Roth]{Oliver Roth}  
\address{Department of Mathematics, University of W\"urzburg, 97074, W\"urzburg, Germany}
\email{ roth@mathematik.uni-wuerzburg.de}   

%
%

\date{}

\begin{abstract}
  We prove several sharp distortion and monotonicity theorems for spherically convex functions  defined on the unit disk  involving geometric quantities such as spherical length, spherical area and total spherical curvature. These results can be viewed as geometric variants of the classical Schwarz lemma for spherically convex functions.
 
\end{abstract}

 \maketitle

\section{Introduction and Results}

Let  $f$ be  a holomorphic function on the unit disk $\D = \{z \in  \mathbb{C} : |z| < 1 \}$, and let $\T =\partial \D$ be the unit circle. In   \cite[p.~165, Problem 309]{polyaszego}, 
G. P{\'o}lya and G. Szeg\H{o} observed  that if $\LeE$ denotes the 
euclidean length of a curve, the function
\begin{equation}\label{length}
r \mapsto \frac{\LeE f (r \T)}{ \LeE (r \T)}= \frac{1}{2 \pi} \int\limits_{0}^{2\pi} \left| f^{\prime}(r e^{it}) \right| dt,
\end{equation}
 is increasing on the interval $(0,1)$. Much more recently,  
it  was proved by Aulaskari \& Chen \cite{aulask} and  by Burckel, Marshall, Minda, Poggi--Corradini \& Ransford \cite{5} that if $\ArE$ denotes the euclidean area of a domain,  
the function 
\begin{equation}\label{area}
r \mapsto \frac{\ArE f(r \D)}{\ArE (r\D)}=\frac{1}{\pi r^2} \ArE f(r \D), 
\end{equation}
is also monotonically increasing. These investigations have since then led to a series of monotonicity results comparing other euclidean geometric and  euclidean potential theoretic quantities of the image $f(r \D)$ with those corresponding to $r\D$. This way, quantitative bounds on the growth behaviour of the image $f(r \D)$ have been established, leading to several distortion theorems. Examples of such euclidean geometric and potential theoretic quantities are the diameter, $n$-th diameter, logarithmic capacity, inner radius, and total curvature; see \cite{ bets1, bets2, 5, kourou1}.

Even more recently, starting with the work of Betsakos \cite{bets4}, many of these euclidean geometric Schwarz--type lemmas have been carried over to the hyperbolic setting.\hide{ and  holomorphic self--maps of the unit disk $\D$ equipped with the hyperbolic metric to  \textit{hyperbolic} analogs of these euclidean geometric results  have been obtained for the case that the unit disk is endowed with the hyperbolic metric.} For instance, in \cite{bets4}, results are proved concerning the hyperbolic-area-radius of $f(r \D)$ and its hyperbolic capacity.
Furthermore, the notion of hyperbolic convexity has led to corresponding monotonicity theorems regarding total hyperbolic curvature, hyperbolic length and area, see \cite{kourou2}.

The purpose of the present work is to establish sharp estimates and monotonicity results for geometric quantities such as spherical length, spherical area and total spherical curvature for  \textit{spherically convex}  functions on the unit disk. Spherically convex functions have been investigated by many authors, including Wirths, K\"uhnau, Ma, Minda, Mej{\'i}a, Pommerenke and others, see the references \cite{kuehnau, mamindasphlinear, sphkcon, sph2, mejiapomsph, wirths} and the references therein. In studying spherical analogs of the aforementioned euclidean and hyperbolic monotonicity results and geometric Schwarz's lemmas one faces a number of difficulties caused by effects of positive curvature as well as several phenomena which are not present at all in the euclidean and hyperbolic situation, and hence a different approach and different tools are required. This paper addresses these issues.

In order to state our results, we first need to recall some basic concepts from spherical geometry. For more details the reader might consult Section \ref{sectionsph} and also \cite{mindabloch} and \cite{app}, for instance. We equip the Riemann sphere $\CC=\mathbb{C} \cup \{\infty\}$ with the spherical metric
$$ \lambda_{\CC}(z) \, |dz|=\frac{|dz|}{1+|z|^2} \, ,$$
the canonical conformal Riemannian metric on $\CC$ with constant Gaussian curvature $+4$. For two points $a,b \in \CC$ which are not antipodal  the  unique spherical geodesic joining $a$ and $b$ is the   smaller arc of the great circle through $a$ and $b$. The meromorphic spherical isometries form the group of rotations of $\CC$ which is explicitly given by
$$\Rot(\CC) := \left\{ e^{ i \theta} \frac{z-a }{1+z \overline{a}} :  \, a \in \mathbb{C}, \, \theta \in \mathbb{R} \right\} \cup \left\{ \frac{e^{i \theta}}{z} : \, \theta \in \mathbb{R} \right\}. $$
For a meromorphic function $f : \D \to \CC$ the spherical derivative
$$ f^{\sharp}(z):=\frac{|f'(z)|}{1+|f(z)|^2}$$ is invariant under postcomposition with  any $T \in \Rot(\CC)$, that is, $(T \circ f)^{\sharp}(z)=f^{\sharp}(z)$.

A domain $\Omega$ on the Riemann sphere $\CC$ is called \textit{spherically convex} if for any two points $a,b \in \Omega$ that are not antipodal the spherical geodesic joining $a$ and $b$  lies entirely in $\Omega$. A meromorphic univalent map $f: \D \to \CC$ is called \textit{spherically convex} if $f(\D)$ is a spherically convex domain in $\CC$.

Our first result provides a sharp upper bound for the spherical area
$$ \ArS f(r \D)=\iint \limits_{r \D} f^{\sharp}(z)^2 \, dA(z) $$
of the image $f(r \D)$ of any spherically convex function $f: \D \to \CC$ in terms
of 
$$ \ArS (r\D):=\iint \limits_{r \D} \frac{dA(z)}{\left(1+|z|^2\right)^2}=\frac{\pi r^2}{1+r^2} \, , $$
the spherical area of the disk $r \D$.

\begin{theorem}[Area Schwarz's Lemma for spherically convex functions] \label{thm:schwarzarea}
  Let $f : \D \to \CC$ be spherically convex. Then
  $$ \ArS f(r \D) \le \ArS(r\D)  \quad \text{ for every } 0<r<1 \, .$$
  Moreover, equality holds for some $0<r<1$ if and only if $f$ is a spherical isometry.
  \end{theorem}

Theorem \ref{thm:schwarzarea} raises the problem whether there exists a corresponding lower bound for the ratio
  \begin{equation}\label{spharea}
\AreS(r) : = \frac{\ArS f(r \D)}{ \ArS (r \D)}, \quad r \in (0,1) \, . 
\end{equation}
Note that (\ref{spharea}) is the spherical analog of the euclidean quantity (\ref{area}).  Since
$$ \lim \limits_{r \to 0+} \AreS(r)=f^{\sharp}(0)^2 \, , $$
a lower bound for $\AreS(r)$  would follow provided one could prove that $\AreS(r)$ is increasing as a function of $r$.

 \begin{theorem}\label{thm:monotarea}
Let $f : \D \to \CC$ be spherically convex. Then $\AreS(r)$ is a strictly increasing function of $r \in (0,1)$, unless $f$ is a spherical isometry in which case $\AreS(r) \equiv 1$.
\end{theorem}

Theorem \ref{thm:monotarea} is a spherical analog of the previously known monotonicity results for
euclidean and hyperbolic area (\cite{aulask,5,kourou2}) mentioned at the beginning.

\begin{corollary}\label{lowerboundarea}
  Let $f : \D \to \CC$ be a spherically convex function. Then 
  \begin{equation}\label{lowerboundareaeq:1}
    \ArS f(r \D)  \geq \ArS (r\D)
  f^{\sharp}(0)^2   \quad \text{ for every } 0<r<1 \, .
\end{equation}
  Moreover, equality holds in \eqref{lowerboundareaeq:1} for some $0<r<1$ if and only if $f $ is a spherical isometry.  
\end{corollary}

\begin{remark}[Theorem \ref{thm:monotarea} vs.~Theorem \ref{thm:schwarzarea}]
Clearly, $\ArS f(r \D) \le \pi/2$ for any spherically convex function $f: \D \to \CC$, so Theorem \ref{thm:monotarea} easily implies 
$$ \frac{\ArS f(r \D)}{\ArS (r\D)}=\AreS(r) \le \limsup  \limits_{\rho \to 1-}  \AreS(\rho)  \le \lim \limits_{\rho \to 1-} \frac{\pi/2}{\ArS (\rho \D)}=\lim \limits_{\rho \to 1-} \frac{1+\rho^2}{2 \rho}=1 $$
for every $0<r<1$.
In this sense,  Theorem \ref{thm:schwarzarea} appears as an easy corollary of Theorem \ref{thm:monotarea}. However, the proof of Theorem \ref{thm:monotarea} we give below depends in an  essential way on Theorem \ref{thm:schwarzarea}, so the apparently stronger statement of Theorem \ref{thm:monotarea} is in fact \textit{equivalent} to Theorem \ref{thm:schwarzarea}. 
\end{remark}

In passing, we note that the \textit{proof} of Theorem \ref{thm:monotarea} leads to another sharp lower bound for the spherical area $\ArS f(r \D)$ which is more precise than the one provided by the sharp inequality (\ref{lowerboundareaeq:1}), but  geometrically less pleasing:

\begin{corollary} \label{lowerboundarea2}
  Let $f : \D \to \CC$ be a spherically convex function.  Then 
  $$\ArS f(r \D) \geq\frac{\pi r^2}{1+r^2 f^{\sharp}(0)^2}
  f^{\sharp}(0)^2 \quad \text{ for every } 0<r<1 \, . $$
Moreover, equality holds for any $0<r<1$ if $f$ has the form $f(z)=T (\eta z)$ with $T \in \Rot(\CC)$ and $0<|\eta|\le 1$. 
\end{corollary}

Our next results deal with the spherical length
$$ \LeS f(r \T):= \int \limits_{r \T} f^{\sharp}(z) \, |dz|$$
of the image $f(r \T)$ of the circle $r \T$ under a meromorphic map $f : \D \to \CC$. We 
denote by
$$ \LeS (r\T):=\int \limits_{r \T} \frac{|dz|}{1+|z|^2}=\frac{2 \pi r}{1+r^2} $$
the spherical length of the circle $r \T$.

\begin{theorem}[Length Schwarz's lemma for spherically convex functions] \label{lowerboundlength1}
Let $f:\D \to \CC$ be a spherically convex function. Then 
$$ \LeS (f(r \T)) \ge  \LeS (r \T) \, f^{\sharp}(0) \quad \text{ for every } 0<r<1 \, .$$
   Moreover, equality holds for some $0<r<1$ if and only if $f$ is a spherical isometry. 
   \end{theorem}

   \begin{remark}\label{rem:schwarzlength}
An \textit{upper} bound for $\LeS f(r \T)$ for  spherically convex functions $f : \D \to \CC$ is 
 $$  \LeS f(r \T) \le \frac{2\pi r}{1-r^2} f^{\sharp}(0)  \quad \text{ for every } 0<r<1 \, .$$ 
\end{remark}

Similar to Corollary \ref{lowerboundarea2} there is also a more precise, but geometrically less natural lower bound for spherical length, which follows from Corollary \ref{lowerboundarea2} in conjunction with the isoperimetric inequality.

\begin{theorem}\label{lowerboundlength2}
Let $f:\D \to \CC$ be  spherically convex. Then 
$$ \LeS (f(r \T)) \ge  \frac{2\pi r f^{\sharp}(0)}{1+r^2 f^{\sharp }(0)^2}  \quad \text{ for every } 0<r<1 \, .$$
   Moreover, equality holds for any $0<r<1$ if $f$ has the form $f(z)=T( \eta z)$ with $T \in \Rot(\CC)$ and $0<|\eta|\le 1$. 
   \end{theorem}

Theorem \ref{thm:monotarea}   raises  the question whether the ratio 
\begin{equation}\label{sphlength}
\LenS(r) : = \frac{\LeS f(r \T)}{ \LeS (r \T)} 
\end{equation}
is monotonically increasing as a function of $r$.
 Note that (\ref{sphlength}) is the spherical analog of the quantity (\ref{length}). 
 While we cannot offer a full answer, we shall now show that 
such a monotonicity result does hold for spherically convex functions $f : \D \to \CC$ which are 
\textit{centrally normalized}:
$$f(z)= \alpha z+ a_3 z^3 + ..., \quad z \in \D, $$
where $$\alpha = \max_{z \in \D} \left(1-|z|^2\right) f^{\sharp}(z).$$
The important additional assumption is that $f''(0) = 0$.
The notion of \textit{central normalization} and the insight of its relevance in the study of spherically convex function is due to Mej{\'i}a and Pommerenke \cite{mejiapomsph} building on earlier work of Minda and Wright \cite{MW}, Chuaqui and Osgood \cite{CO} and Chuaqui, Osgood and Pommerenke \cite{COP}. According to \cite[Theorem 4]{mejiapomsph}, for any spherically convex function $f$ there is always a unit disk automorphism $\psi$ and  a rotation $T \in \Rot(\CC)$ such that $T \circ f \circ \psi$ is centrally normalized.

 \begin{theorem}\label{thm:monotlength}
  Let $f: \D \to \CC$ be a centrally normalized spherically convex function. Then $\LenS(r)$ is a strictly increasing function of $r \in (0,1)$, unless $f$ is a spherical isometry in which case $\LenS(r) \equiv 1$.
  \end{theorem}

Theorem \ref{thm:monotlength} is a spherical analog of the previously known monotonicity results for euclidean and hyperbolic length (\cite{polyaszego,kourou2}).

\medskip

  In addition to spherical length and spherical area another important geometric quantity in spherical geometry is the total spherical curvature
  $$ \int \limits_{\gamma} \kappa_s(w,\gamma) \lambda_{\CC}(w) \, |dw|$$
  of a curve $\gamma$, see Section \ref{sectionsph} and \cite{app}. Roughly speaking, total spherical curvature measures how much the curve $\gamma$ diverges from being a spherical geodesic. We consider 
the ratio
   $$\Phi_s(r) : = \frac{\displaystyle\int\limits_{f(r \T)}  \kappa_s (w, f(r \T))\, |dw|}{\displaystyle\int\limits_{r \T}  \kappa_s (z, r \T) \, |dz|}. $$
and  prove the following monotonicity property.

\begin{theorem}\label{thm:monotcurv}
Let $f : \D \to \CC$ be a centrally normalized spherically convex function. Then $\Phi_s(r)$ is a strictly increasing function of $r \in (0,1)$, unless  $f$ is a spherical isometry in which case $\Phi_s(r) \equiv 1$. 
\end{theorem}

One of the crucial ingredients of the proofs of the above theorems is a basic result  from \cite[Theorem 4]{mamindasphlinear} which guarantees  that a meromorphic univalent function $f$ in $\D$ is spherically convex if and only if the auxiliary function
\begin{equation}\label{sphconfun}
h_f(z):= \RE \left\{ 1 + \frac{z f^{\prime \prime }(z)}{f^{\prime}(z)}-\frac{2z f^{\prime}(z) \overline{f(z)}}{1+\left| f(z) \right|^2} \right\} 
\end{equation}
has the property that
$$ h_f(z) \geq 0 \quad \text{ for every } z \in \D \, .$$ This characterization of spherical convexity has an elegant geometric interpretation in terms of the spherical curvature $\kappa_s(f(z), f(r \T)$ of the curve  $f(r \T)$ at the point $f(z)$, $|z|=r$,  since
$$h_f(z)=\kappa_s (f(z),f(r \T)) f^{\sharp}(z) |z| \,,$$ see (\ref{scurvimage}) below, so a meromorphic univalent function $f$ in $\D$ is spherically convex if and only if $$\kappa(f(z),f(r \T)) \geq 0 \quad  \text{ for all } |z|=r \text{ and all }  0<r<1 \,  .$$
For further information on spherical convexity and spherically convex functions we refer to Section \ref{sectionsph} and also to \cite{hypmetsph, mamindasphlinear,  sphkcon, sph2,  mejiapomsph, app,  sugawa} as well as to the recent work \cite{kelgiannis2}, where monotonicity results are proved regarding the elliptic-area-radius of $f( r \D)$ and condenser capacity. Other variants of the Schwarz lemma for meromorphic functions can be found e.g.~in \cite{Dubinin, Solynin}.

\medskip

The paper is structured in the following way. In Section \ref{sectionsph} we recall a number of  basic facts about  spherical geometry and spherical convexity which are necessary for our investigations, including the  spherical Gauss--Bonnet Theorem and the spherical isoperimetric inequality.
In Section \ref{hf} we study the auxiliary function $h_f$ defined in (\ref{sphconfun}) and give a new characterization of spherical convexity as  well as establishing a sharp lower bound for the integral means of $h_f$. A corresponding \textit{pointwise} sharp lower estimate for $h_f$ has been given
by Mej{\'i}a and Pommerenke in their important work \cite{mejiapommerenke} on the Schwarzian derivative for spherically convex functions. While the estimate of Mej{\'i}a and Pommerenke is valid only for centrally normalized functions,  our 'integrated' version does hold for any spherically convex function and  possesses a natural geometric significance in terms of total geodesic curvature.\hide{,  and   will be a central tool for our work.}
The spherical Schwarz--type lemmas, Theorem \ref{thm:schwarzarea} and \ref{lowerboundlength1} and Remark \ref{rem:schwarzlength}, are proved in Section \ref{sec:schwarz}. Then attention shifts to monotonicity results for spherically convex functions. 
In Section \ref{monotonarealength} we consider spherical area and prove Theorem \ref{thm:monotarea} as well as Corollary \ref{lowerboundarea2} and Theorem \ref{lowerboundlength2}.
The monotonicity of spherical length  (Theorem \ref{thm:monotlength}) and of 
total spherical curvature (Theorem \ref{thm:monotcurv}) for centrally normalized spherically convex functions 
is established in Section \ref{monotoncurv}. In a final Section \ref{examples} we illustrate by examples that spherical convexity is a basic requirement for Theorems \ref{thm:monotarea}, \ref{thm:monotlength} and \ref{thm:monotcurv} and that central normalization is a necessary hypothesis for Theorem \ref{thm:monotcurv}.

\section{Spherical Convexity - Gauss Bonnet formula - Isoperimetric Inequality}\label{sectionsph}

Suppose $f: \D \to \CC$ is a meromorphic univalent function and $f(\D)$ is a hyperbolic domain in $\CC$. 

\begin{lemma}\cite[Theorem 1]{hypmetsph}\label{mu}
The spherical density $\left(1-|z|^2\right) f^{\sharp}(z)$ is a superharmonic function on $\D$ if and only if $f(\D)$ is a spherically convex domain.
\end{lemma}

\begin{lemma}\cite[p.288]{hypmetsph}\label{sphderupper}
If $f$ is spherically convex on $\D$, then $\left(1-|z|^2\right) f^{\sharp}(z) \leq 1$ for every $z \in \D$. 
Equality holds  for some $z \in \D$ if and only if $f$ maps $\D$ onto a hemisphere and $z$ is the spherical center of the hemisphere. 
\end{lemma}

\begin{proposition}\cite[Theorem 4]{mamindasphlinear} 
Let $f :\D \to \CC$ be a meromorphic univalent function. Then $f$ is spherically convex if and only if 
 $$ h_f(z) = \RE \left\{  1 + \frac{z f^{\prime \prime}(z)}{f^{\prime}(z)} - 2 \frac{z f^{\prime}(z) \overline{f(z)}}{1+|f(z)|^2} \right\}  \geq 0,  \quad z \in \D.$$
\end{proposition}

 Mej{\'i}a and Pommerenke,  see \cite[(3.14)]{mejiapommerenke}, have proved that 
 for any \textit{centrally normalized} spherically convex  function $f$,
 \begin{equation}\label{lowerhf}
 h_f(z) \geq \frac{1-|z|^2}{1+|z|^2} \, .
 \end{equation}
In fact, it is not difficult for the reader to convince himself that equality can hold in (\ref{lowerhf}) for some $z \in \D$ if and only if $f$ is a spherical isometry. 

For a geometric interpretation of spherical convexity, we briefly discuss the notion of \textit{spherical curvature}. A curve $\gamma$ is said to have  spherical curvature $\kappa_s(z, \gamma) $ equal to $0$ at any of its points if and only if it is a spherical geodesic. 

Let $\gamma :z=z(t)$ be a $C^2$ curve on $\CC$ with everywhere non-vanishing tangent. The spherical curvature of $\gamma$ at $z(t)$ is $$\kappa_s (z(t), \gamma) \lambda_{\CC} (z(t))= k(z(t), \gamma) - \IM \left\{ \frac{2 \overline{z(t)} z^{\prime}(t)}{\left(1+|z(t)|^2\right) |z^{\prime}(t)|} \right\}, $$
where $\kappa (z(t), \gamma) $ is the euclidean curvature of $\gamma$ at $z(t)$. 

It can easily be calculated that the spherical curvature of $r \T$ at a point $z\in r\T$ is equal to 
\begin{equation}\label{sphcurvrT}
\kappa_s(z, r \T) = \frac{1-r^2}{r}. 
\end{equation}

\begin{proposition}\cite[Theorem 3]{mamindasphlinear}
If $\Omega \subset \mathbb{P}$ has $\mathcal{C}^2 $ smooth boundary and $\Omega $ is spherically convex, then for all $z \in \partial \Omega $, $\kappa_s(z, \partial \Omega) \geq 0$.  

\end{proposition}

\begin{proposition}\cite[Theorem 2]{mamindasphlinear}\label{formcurv}
Suppose $f:\D \to \CC$ is a meromorphic univalent function and $\gamma : z=z(t)$ is a $\mathcal{C}^2$ curve in $\D$. Then 
\begin{eqnarray}\label{sphcurvaturef}
  \kappa_s(f(z), f \circ \gamma) \left(1-|z|^2\right) f^{\sharp}(z) & =& \\ & &  \hspace*{-4cm} \nonumber  
   \kappa_h(z, \gamma) - \left(1-|z|^2\right)  \IM \left\{ \left(2\frac{\bar{z}}{1-|z|^2}- \frac{f^{\prime \prime }(z)}{f^{\prime }(z)} + \frac{ 2 f^{\prime }(z) \overline{f(z)}}{1+|f(z)|^2} \right) \frac{z^{\prime}(t)}{|z^{\prime}(t)|} \right\}, 
 \end{eqnarray}
where $ \kappa_h$ denotes the hyperbolic curvature and $$\kappa_h (z, \gamma) = \left(1-|z|^2\right) \kappa(z,\gamma) +2 \IM \left\{
\frac{\overline{z(t)} z^{\prime }(t)}{|z^{\prime}(t)|}\right\}.  $$  
\end{proposition}

Let $f:\D \to \CC$ be a meromorphic univalent function. For the definition of the function $\Phi_s(r)$, as stated in the Introduction, we will need the spherical curvature of $f(r \T)$. Therefore, according to \eqref{sphcurvaturef} 
\begin{eqnarray*}
\kappa_s(f(z), f(r \T)) \left(1-r^2\right) f^{\sharp}(z) & = &\frac{1+r^2}{r}  - \frac{1-r^2}{r} \IM \left\{ i \left[ 2 \frac{r^2}{ 1-r^2} - \frac{z f^{\prime \prime }(z)}{f^{\prime }(z)} + \frac{ 2 z f^{\prime }(z) \overline{f(z)}}{1+|f(z)|^2} \right] \right\} \\
& = &\frac{1-r^2}{r}  +  \frac{1-r^2}{r}\RE \left\{ \frac{z f^{\prime \prime }(z)}{f^{\prime }(z)} - \frac{ 2 z f^{\prime }(z) \overline{f(z)}}{1+|f(z)|^2} \right\}  \\
&=&\frac{1-r^2}{r} h_f(z) 
\end{eqnarray*}
for $z=re^{it}, r \in (0,1), t \in [0,2\pi]$ and $\kappa(z,\gamma)  = \frac{1}{r}$. 
Hence 
\begin{equation}\label{scurvimage} 
\kappa_s\left(f(z), f(r \T)\right) = \frac{h_f(z)}{|z|f^{\sharp}(z)}, 
\end{equation}
where $z =r e^{it}$. The total spherical curvature is a geometric quantity that measures how much a curve diverges from being spherically convex. 
From \eqref{sphcurvrT}, the total spherical curvature of $r \T$ is equal to  
\begin{equation}\label{totalr}
\int\limits_{r \T} \kappa_s(z, r \T) \lambda_{\CC}(z) \, |dz| =2 \pi \frac{1-r^2}{1+r^2}
\end{equation}
and from \eqref{scurvimage}, the total spherical curvature of $f(r \T)$ is 
\begin{equation}\label{totalf}
\int\limits_{f(r \T)}  \kappa_s(w, f(r \T) \lambda_{\CC}(w) \, |dw|  = \int\limits_{r \T} \kappa_s(f(z), f (r \T)) f^{\sharp}(z) \, |dz| \underset{\eqref{scurvimage}}{=} \int\limits_{0}^{2 \pi} h_f(r e^{it}) \, dt. 
\end{equation} 

For more information on spherical convexity and spherical curvature, the reader may refer to \cite{mamindasphlinear, sphkcon, app}. 

In the proof of Theorem \ref{thm:monotcurv}, we will use the Gauss-Bonnet formula in the following form, see \cite[Theorem 6.5]{spivak}. 
Let $M$ be an oriented two-dimensional Riemannian manifold with Gaussian curvature $K$ and volume element $d A$. Let $N \subset M $ be a compact two-dimensional manifold-with-boundary which is diffeomorphic to a subset of $\mathbb{R}^2$ and whose boundary is connected. Let $ds$ be the volume element of $\partial N $ and let $\kappa$ be the signed geodesic curvature of $\partial N$. 
Then 
\begin{equation}\label{GaussBonnet}
\int\limits_{N} K \, dA + \int\limits_{\partial N} \kappa \, ds = 2 \pi.
\end{equation}

The Riemann sphere $\CC$ endowed with the  spherical metric is a two-dimensional Riemannian manifold of constant Gaussian curvature equal to $4$. 
If $\Omega$ is a hyperbolic domain in $\CC$, the Gauss-Bonnet formula \eqref{GaussBonnet} takes the form 
\begin{equation}\label{GBS}
4 \ArS (\Omega) + \int\limits_{\gamma} \kappa_s (z, \gamma) \lambda_{\CC}(z) \, |dz| = 2 \pi, 
\end{equation}
where $\gamma$ is the boundary of $\Omega$ assuming that it is a smooth, simple and closed curve in $\CC$. 

Applying the Gauss-Bonnet formula \eqref{GBS} to $f( r \D)$ viewed as a two-dimensional manifold with boundary on the Riemann surface of $f$, we obtain
\begin{equation} \label{sphcurvGB}
\int\limits_0^{2 \pi} h_f(r e^{it}) \, dt = 2 \pi - 4 \ArS  f(r \D),
\end{equation}
where $\ArS f(r \D) $ is the spherical area of $f(r \D)$.

Last but not least, in order to prove lower bounds for the spherical length, we use the isoperimetric inequality of spherical geometry; see \cite{osserman}. Suppose $D$ is a simply connected smooth subdomain of $\CC$. Then 
\begin{equation}\label{isoperimetricineq} 
\LeS(\partial D)^2 \geq 4 \pi \ArS (D) - 4 \ArS(D)^2. 
\end{equation}

\section{The Function $h_f$}\label{hf}
For our purposes the following characterization of spherically convex functions in terms of
the function $$h_f(z)  = \RE \left\{ 1+  z \frac{f^{\prime \prime }(z)}{f^{\prime}(z)} - \frac{2z \overline{f(z)} f^{\prime}(z)}{1+|f(z)|^2} \right\}$$ turns out to be useful.

\begin{theorem}\label{hsuperharmonic}
  Let $f: \D \to \CC$ be a meromorphic univalent function. Then
  \begin{equation} \label{eq:h}
\De h_f(z) = -8 f^{\sharp}(z)^2\, h_f(z) \, .
    \end{equation}
In particular,  $f$ is spherically convex if and only if $h_f$ is superharmonic on $\D$. In this case, $h_f$ is strictly superharmonic, so $\Delta h_f<0$ in $\D$.
\end{theorem}


\begin{proof}
 Let  $u(z): =\log f^{\sharp} (z)$, so $f^{\sharp} (z)= e^{u(z)}$.
Taking the derivative  of $u$ with respect to $z$, we obtain $$\partial_z u(z) = \frac{1}{f^{\sharp} (z)} \partial_z f^{\sharp} (z)= \frac{f^{\prime \prime }(z)}{2 f^{\prime}(z)} - \frac{ \overline{f(z)} f^{\prime}(z)}{1+|f(z)|^2} . $$
This implies that the real part of  $v(z) : = 1+2z \partial_z u(z)$ is exactly  $h_f$.
Now  $u$ is a solution of the Liouville equation
$$\Delta u(z) = -4 e^{2u(z)}. \, $$
Hence the Laplacian of $v$ is given by  
\begin{align*}
    \De v(z) &= 4 \partial_z \overline{\partial}_z (1+2 z \partial_z u(z))   = 8 \partial_z (-z e^{2u(z)}) \\
    &= -8 e^{2u(z)} -8z e^{2u(z)} \cdot \left(2 \partial_z u(z)\right) = -8 e^{2 u(z)} v(z) \, .
\end{align*}
Taking the real part gives  (\ref{eq:h}). In particular,
 $\De h_f(z) \leq 0$ if and only if $h_f(z) \ge 0$, so $f$ is a spherically convex function if and only if $h_f$ is superharmonic. Suppose $f$ is spherically convex. If $h_f(z_0)=0$ for some $z_0 \in \D$, then $h_f$ would attain its global minimum at $z_0$ and hence would be constant $0$  by the minimum principle for superharmonic functions. But $h_f(0)=1$. This contradiction shows that $h_f$ is strictly positive on $\D$. Since $f$ is univalent, $f^{\sharp}$ never vanishes, and hence $h_f$ is strictly superharmonic. 
\end{proof}

Theorem \ref{hsuperharmonic} implies that for any spherically convex function $f : \D \to \CC$ the integral means
$$  \frac{1}{2\pi}\int\limits_0^{2 \pi} h_f(r e^{it}) \, dt$$
are strictly decreasing and log--concave. The following result provides the sharp lower bound for these integral means.

\begin{theorem} \label{hintbounds}
Let $f: \D \to \CC$ be spherically convex. Then for any $r \in (0,1)$
\begin{equation} \label{eq:totalcurv}
  \frac{1}{2\pi}\int\limits_0^{2 \pi} h_f(r e^{it}) \, dt \ge \frac{1-r^2}{1+r^2}\, .
\end{equation}
For fixed $r \in (0,1)$ equality holds in (\ref{eq:totalcurv}) if and only if $f$ is a spherical isometry.
\end{theorem}

Theorem \ref{hintbounds} is an integrated version of the Mej{\'i}a--Pommerenke inequality (\ref{lowerhf}), but with the additional benefit that we do not need to assume central normalization.
The estimate (\ref{eq:totalcurv}) has a  natural geometric interpretation by observing that the integral expression is precisely the \textit{normalized total spherical curvature} of $f(r \T)$, while the right-hand side is 
the normalized total spherical curvature of the circle $r \T$, see Section \ref{sectionsph}.

\begin{proof} Since $h_{T \circ f}=h_f$ for any $T \in\Rot(\CC)$ we may assume $f(0)=0$.
Fix $t \in [0,2 \pi]$. We apply a beautiful idea from \cite[Theorem 1 \& (3.13)]{mejiapomsph}, namely that the function 
$$p_t(z): = 1  + z\frac{f^{\prime \prime}(z)}{f^{\prime}(z)} -2  \frac{z f^{\prime}(z) \overline{f(e^{2it}\bar{z})}}{1+f(z) \overline{f(e^{2it}\bar{z})}}$$ belongs to the Carath{\'e}odory class
$$\mathcal{P}:=\{ p : \D \to \mathbb{C} \text{ holomorphic} \, : \, p(0)=1, \RE p>0\} \, , $$
since $\RE p_t(r e^{it}) = h_f(r e^{it})$. Hence $p_t(e^{it} z)$ also belongs to $\mathcal{P}$. In view of the convexity and compactness of $\mathcal{P}$, the function 
$$P(z): =\frac{1}{2\pi} \int\limits_0^{2 \pi} p_t(e^{it}z) \, dt= 1 + c_1 z+ c_2 z^2+ ...  $$
also lies in $\mathcal{P}$.
We claim that $c_1=0$. In order to see this, recall that by assumption $f(0)=0$, so $f$ is \textit{holomorphic} in $\D$ with $f'\not=0$. Then $z w f''(zw)/f'(zw)$ is a holomorphic function of $w$ in a neighborhood of the closed unit disk, so the mean value property implies
 $$\frac{1}{2\pi} \int\limits_{0}^{2 \pi} \frac{e^{it} z f^{\prime \prime}(e^{it}z)  }{f^{\prime} (z)} \, dt  =  \frac{zw f^{\prime \prime}(zw) }{f^{\prime}(zw) } \bigg|_{w=0} =0 \, .$$
Hence 
$$ P(z) = 1 -\frac{1}{\pi} \int\limits_0^{2 \pi} \frac{e^{it} z f^{\prime}(e^{it}z ) \overline{f(e^{it} \bar{z})}}{1+ f(e^{it}z) \overline{f(e^{it}\bar{z})} } \, dt  =1 -2 \, |f^{\prime}(0)|^2 z^2 + \smallO(z^3),$$
using once more that $f(0)=0$.  Thus  the function $(P-1)/(P+1) : \D \to \D$ has a zero of order at least two at $z=0$, so the Schwarz lemma implies 
$$\left| \frac{P(z)-1}{P(z)+1} \right| \leq |z|^2, \quad z \in \D\, , $$
with equality at one point if and only $P$ has the form
$$ P(z)=\frac{1+\omega z^2}{1-\omega z^2} =1+2 \omega z^2+\ldots $$
for some $|\omega|=1$. We conclude that
$$\frac{1-r^2 }{1+r^2} \leq \RE P(r) \leq \frac{1+r^2}{1-r^2} \qquad \text{ for any } r \in (0,1) \, .$$
Equality for the left inequality holds if and only if $\omega=-1$ resp.~$|f'(0)|=1$. By Lemma \ref{sphderupper}
 this is the case if and only if $f(z)=\eta z$ for some $|\eta|=1$.
\end{proof}

The Gauss--Bonnet formula \eqref{sphcurvGB}  provides us with the following result. 

\begin{proposition}\label{connection*}
  Let $f:\D \to \CC$ be a meromorphic univalent function. Then
   \begin{equation}\label{connection}
     \int\limits_{0}^{2 \pi} f^{\sharp}(re^{it})^2 \,dt = \frac{2}{r^2} \iint\limits_{r \D} h_f(z) f^{\sharp}(z)^2 \,dA(z)
     \end{equation}
  for any $r \in (0,1)$.
\end{proposition}

\begin{proof}
  Taking the derivative w.r.t.~$r$ in the Gauss--Bonnet formula
 \eqref{sphcurvGB}, we obtain
\begin{equation} \label{eq:derhf}
\frac{\partial}{\partial r} \left( \int\limits_0^{2\pi} h_f(r e^{it}) \, dt \right)= -4 \frac{\partial}{\partial r} \iint\limits_{r \D} f^{\sharp}(z)^2 \, dA(z) = -4 r \int\limits_{0}^{2 \pi} f^{\sharp}(re^{it})^2 \, dt \, .
\end{equation}
By  Green's formula, we also see that
$$\frac{\partial}{\partial r} \left( \int\limits_{0}^{2\pi} h_f(r e^{it}) \, dt \right) = \frac{1}{r} \iint\limits_{r \D} \De h_f(z) \, dA(z)\, .$$
Together with (\ref{eq:derhf}) this yields
 $$
 \frac{1}{r} \iint\limits_{r \D} \De h_f(z) \, dA(z) = -4 r \int\limits_{0}^{2 \pi} f^{\sharp}(re^{it})^2 \, dt \, .
 $$
 Since $\Delta h_f=-8 (f^{\sharp})^2 h_f$ by Theorem \ref{hsuperharmonic} we see that (\ref{connection}) holds.
\end{proof}

\section{Proofs of the spherical Schwarz lemmas} \label{sec:schwarz}

\begin{proof}[Proof of  Theorem \ref{thm:schwarzarea}]
Let $f : \D \to \CC$ be a spherically convex function and $0<r<1$. 
The Gauss--Bonnet formula (\ref{sphcurvGB}) 
and Theorem \ref{hintbounds} imply
$$ 4 \ArS f(r \D)=\int \limits_{0}^{2\pi} \left(1- h_f(r e^{it}) \right) \, dt \le 2\pi -2 \pi \frac{1-r^2}{1+r^2} = \frac{4\pi r^2}{1+r^2} = 4 \ArS (r \D)$$
whith equality for some $r \in (0,1)$ if and only if $f \in \Rot(\CC)$. 
\end{proof}

For the proof of Theorem \ref{lowerboundlength1} we first derive an auxiliary lemma. 

\begin{lemma} \label{lem}
  Let $f : \D \to \CC$ be a spherically convex function. Then the integral mean
$$ \frac{1}{2\pi} \int \limits_0^{2\pi} \log \left[ \left(1+r^2\right) f^{\sharp}(r e^{it})\right] \, dt $$
  is strictly increasing as a function of $r$ unless $f \in \Rot(\CC)$.
  \end{lemma}

  \begin{proof}
  It is easy to prove that 
$$r \frac{\partial }{\partial r} \log \left[ \left(1+r^2\right) f^{\sharp} (r e^{it}) \right]  = h_f(r e^{it}) -\frac{1-r^2}{1+r^2}. $$
The result therefore follows from Theorem \ref{hintbounds}. \end{proof}

\begin{proof}[Proof of Theorem \ref{lowerboundlength1}]
Let $f : \D \to \CC$ be a spherically convex function and $0<r<1$. Then 
\begin{eqnarray*}
  \frac{1}{2\pi} \LeS f(r \T) (1+r^2)  &=& r \left(1+r^2\right) \frac{1}{2\pi} \int \limits_{0}^{2\pi} f^{\sharp}(r e^{it}) \, dt \\
                                        &=&  \frac{r}{2\pi} \int \limits_0^{2\pi} \exp \left( \log \left[ \left(1+r^2\right) f^{\sharp}(r e^{it}) \right] \right) \, dt \\
  & \ge & r \cdot \exp \left( \frac{1}{2\pi} \int \limits_{0}^{2\pi} \log  \left[ \left(1+r^2\right) f^{\sharp}(r e^{it}) \right]  \, dt \right) \\ & \ge & r  \cdot\exp \left( \frac{1}{2\pi} \int \limits_{0}^{2\pi} \log  \left[ f^{\sharp}(0) \right]  \, dt \right)\\ &=& r f^{\sharp}(0) \, , 
                                           \end{eqnarray*}
                                           where we have first used Jensen's inequality and then Lemma \ref{lem}.
                                           Equality holds if and only if $f \in \Rot(\CC)$. Hence
                                           $$ \LeS f(r \T) \ge \frac{2\pi r}{1+r^2} f^{\sharp}(0) =\LeS (r \T) f^{\sharp}(0) \, , $$
                                           and  equality holds if and only if $f \in \Rot(\CC)$. 
                                           \end{proof}

                                           \begin{proof}[Proof of Remark \ref{rem:schwarzlength}]
                             
 We apply Lemma \ref{mu}.  
 Let us  denote by $\Le_h $ the length of a curve in $\D$ with respect to the hyperbolic metric
$$ \frac{|dz|}{1-|z|^2}$$
 of the unit disk. 
Set $$L(r):= \frac{\LeS f(r \T)}{\Le_h (r \T)} = \frac{1-r^2}{2 \pi} \int\limits_{0}^{2 \pi} f^{\sharp}(r e^{it}) \, dt= \frac{1}{2 \pi} \int\limits_{r \T} \left(1-|z|^2\right) f^{\sharp}(z)\,  |dz| . $$
According to Lemma \ref{mu}, $\left(1-|z|^2\right) f^{\sharp}(z)$ is a superharmonic function on $\D$ if and only if $f(\D)$ is a spherically convex domain. Since $f$ is a spherically convex function, $L(r)$ is the mean value of a superharmonic function and hence decreasing. Therefore  
$$ L(r) \leq \lim_{r \to 0^{+}} L(r) = f^{\sharp}(0) \, , $$
and hence 
$$  \LeS f(r \T) \leq \frac{2 \pi r }{1-r^2}f^{\sharp}(0) $$
for all $r \in (0,1)$.
\end{proof}

\section{Monotonicity of Spherical Area and Length}\label{monotonarealength}
Suppose $f : \mathbb{D} \to \CC$ is a spherically convex function. 

\begin{proof}[Proof of Theorem \ref{thm:monotarea}]
With the use of the Cauchy-Schwarz inequality, we obtain a lower bound for the derivative 
\begin{eqnarray}
\frac{\partial}{\partial r} \ArS f(r \D)&=& r \int \limits_{0}^{2 \pi} f^{\sharp }(r e^{it})^2 dt \nonumber \geq  \frac{r}{2\pi} \left(\int \limits_{0}^{2 \pi} f^{\sharp }(r e^{it}) dt \right)^2 \nonumber =\frac{1}{2\pi r } \LeS f(r \T)^2. 
\end{eqnarray}
Utilizing the isoperimetric inequality \eqref{isoperimetricineq}, it follows 
\begin{equation}\label{isoperimboundeq:1}
\frac{\partial}{\partial r} \ArS f(r \D) \geq \frac{2}{\pi r} \left( \pi \ArS f(r \D) - \ArS f(r \D)^2 \right). 
\end{equation}
We aim to find a lower bound of the derivative of the function $\AreS (r) = \frac{1+r^2}{\pi r^2} \ArS f(r \D) $ in order to prove its monotonicity. Accordingly, we compute with the help of (\ref{isoperimboundeq:1})
\begin{equation}\label{differineq:1}\begin{split}
\AreS^{\prime} (r) &= -\frac{2}{\pi r^3} \ArS f (r \D) + \frac{1+r^2}{\pi r^2} \frac{\partial }{\partial r} \ArS f(r \D) \\
&\geq\frac{2}{\pi r^3}  \ArS f (r \D) \left(-1 +\frac{1+r^2}{\pi} \left(\pi - \ArS f (r \D)\right) \right) \\
&=\frac{2}{\pi r}  \ArS f (r \D) \left(1- \frac{1+r^2}{\pi r^2}  \ArS f (r \D) \right) \geq 0
\end{split}
\end{equation}
due to Theorem \ref{thm:schwarzarea}. As a result, $\AreS(r)$ is increasing in $(0,1)$.
Now, let us assume that $\AreS^{\prime} (\xi) = 0$ for some $\xi \in (0,1)$. Then $$\ArS f(\xi \D) = \frac{\pi \xi^2}{1+\xi^2} = \ArS (\xi \D)$$
and according to Theorem \ref{thm:schwarzarea}, $f$ is a spherical isometry. On the contrary, if $f$ is a spherical isometry, then  $\ArS f(r \D) = \ArS (r \D) = \pi r^2/(1+r^2)$ for all $r \in (0,1)$. Hence $\AreS(r)$ is constant and equal to $1$. 
In summary, $\AreS(r)$ is a strictly increasing function of $r \in (0,1)$, unless $f$ is a spherical isometry in which case $\AreS(r) \equiv 1$. 
\end{proof}

\begin{proof}[Proof of Corollary \ref{lowerboundarea2}] If $f \in \Rot(\CC)$, then there is nothing to prove in view of Corollary \ref{lowerboundarea}.
  Suppose that $f:\D \to \CC$ is not a spherical isometry. As we see from  \eqref{differineq:1}, the function
  $x(r):= \AreS (r)$ satisfies the differential inequality
$$x^{\prime}(r) \geq  \frac{2r}{1+r^2} x(r) \left(1-x(r)\right) \, .$$
Since we assume $f \not\in\Rot(\CC)$, we have $x(r)<1$ for any $0 \le r<1$ by Theorem \ref{thm:monotarea}, and hence
$$ \int \limits_0^{R} \frac{x^{\prime}(r)}{x(r) \left(1-x(r)\right)} \, dR \geq  \int \limits_0^{R} \frac{2r}{1+r^2} \,  dR \, .$$
By elementary integration and reorganization of terms, 
 we are led to $$  x(R) \geq \frac{1+R^2}{1+x(0) R^2} x(0),\, $$
for every $R \in (0,1)$. However $x(0) = \lim\limits_{r \to 0} \AreS(r)= f^{\sharp }(0)^2$ and thus, replacing $R$ by $r$,
$$\AreS(r) \geq \frac{1+r^2 }{1+r^2 f^{\sharp }(0)^2} f^{\sharp }(0)^2\, , $$
which is equivalent to 
$$\ArS f(r \D) \geq \frac{\pi r^2}{1+r^2 f^{\sharp }(0)^2} f^{\sharp }(0)^2\, . $$
If $f$ has the form $f(z)=T(\eta z)$ for some $T \in \Rot(\CC)$ and $0<|\eta|\le 1$, then writing $f_{\eta}(z):=\eta z$, we have
$f^{\sharp}(0)=|\eta|$ and 
$$ \ArS f(r \D)=\ArS f_{\eta}(r\D)=\ArS (r|\eta| \D)=\frac{\pi r^2}{1+r^2 f^{\sharp }(0)^2} f^{\sharp }(0)^2 \, .$$
\end{proof}

\begin{proof}[Proof of Theorem \ref{lowerboundlength2}] Suppose that $f:\D \to \CC$ is a spherically convex function. It is then clear that $\ArS f(\D) \leq \pi/2$.
According to the isoperimetric inequality \eqref{isoperimetricineq}, we have 
$$\LeS f(r \T)^2 \ge 4 \pi \ArS f(r \D) -4   \ArS f(r \D)^2.$$
 The expression $4 \pi \ArS f(r \D) -4   \ArS f(r \D)^2$ is increasing with respect to $\ArS f(r \D)$ on the interval $(0,\pi/2)$, and we obtain  from Corollary \ref{lowerboundarea} that
\begin{equation*}
  \LeS f(r \T)^2  \ge   \frac{4 \pi^2 r^2}{1+r^2 f^{\sharp }(0)^2} f^{\sharp }(0)^2 \left(1 - \frac{r^2}{1+r^2 f^{\sharp }(0)^2 } f^{\sharp}(0)^2 \right)   =
 \frac{4 \pi^2 r^2 }{(1+r^2 f^{\sharp }(0)^2)^2}  f^{\sharp}(0)^2
\end{equation*}
for all $r \in (0,1)$. The proof of the equality statement is identical to the corresponding proof for Corollary \ref{lowerboundarea2} and will be omitted.
\end{proof}

\begin{remark}
The same proof as for Theorem \ref{lowerboundlength2} but using Theorem \ref{thm:schwarzarea} instead of Corollary \ref{lowerboundarea}
produces the inequality
  \begin{equation} \label{eq:lowerboundlengthnew}
    \LeS f(r \T) \ge \LeS (r \T) \, f^{\sharp}(0) \left( 1+r^2 \left(1-f^{\sharp}(0)^2 \right) \right) \quad \text{ for every } 0<r<1 
  \end{equation}
  with equality for some $0<r<1$ if and only if $f$ is a spherical isometry. Since $f^{\sharp}(0) \le 1$ with equality if and only if $f \in \Rot(\CC)$ the estimate (\ref{eq:lowerboundlengthnew}) is slightly more precise than Theorem \ref{lowerboundlength1}.
\end{remark}

\section{Monotonicity for length and total curvature} \label{monotoncurv}

\begin{proof}[Proof of Theorem \ref{thm:monotlength}] Let $f :\D \to \CC$ be a centrally normalized spherically convex function.
Then
$$ \LenS^{\prime}(r) = \frac{\partial }{\partial r} \left( \frac{1+r^2}{2 \pi} \int\limits_0^{2 \pi} f^{\sharp} (r e^{it}) \, dt \right) =  \frac{1}{2\pi} \int\limits_0^{2 \pi} \frac{\partial }{\partial r} \left[ \left(1+r^2\right) f^{\sharp} (r e^{it}) \right] \, dt \, .$$
Now it is easy to see (cf.~also  \cite[p.169]{mejiapommerenke}) that 
$$\frac{\partial }{\partial r} \left( (1+r^2) f^{\sharp} (r e^{it}) \right)  = \frac{f^{\sharp}(re^{it})}{r} \left[\left(1+r^2\right) h_f(r e^{it}) -\left(1-r^2\right) \right]. $$
As a result, 
\begin{equation}
    \LenS^{\prime}(r)= \frac{1}{2\pi r } \int\limits_0^{2 \pi} f^{\sharp}(re^{it}) \left[\left(1+r^2\right) h_f(r e^{it}) -\left(1-r^2\right) \right] dt \geq 0, 
\end{equation}
in view of  \eqref{lowerhf} with equality if and only if $f$ is a spherical isometry, in which case
$f^{\sharp } (z) = 1/ (1+|z|^2)$ and $\LeS f(r \T) = \LeS (r \T)$, so $\LenS(r) \equiv 1$.
\end{proof}

\begin{proof}[Proof of Theorem \ref{thm:monotcurv}]
Let $f: \D \to \CC $ be a spherically convex function which is centrally normalized. 
Following the calculations in \eqref{totalr} and \eqref{totalf}, the ratio of total curvature of $f (r \T)$ to the total curvature of $r \T$ is defined as 
$$\Phi_s(r) = \frac{1+r^2}{2 \pi \left(1-r^2\right)} \int\limits_0^{2 \pi} h_f(r e^{it}) \, dt.$$
If $f$ is a spherical isometry, then clearly $\Phi_s(r) \equiv 1$, so we assume from now on that $f$ is not a rotation. 
Taking the derivative of $\Phi_s$ and using  (\ref{eq:derhf})  we deduce 

$$
  \frac{\pi}{2r} \left(1-r^2\right)^2 \Phi_s^{\prime}(r)=
  \int\limits_0^{2 \pi} h_f(r e^{it}) \, dt -  \left(1-r^4\right) \int\limits_0^{2 \pi} f^{\sharp}(r e^{it})^2 \,  dt\, .$$
Note that from (\ref{eq:derhf}) and
$$\frac{\partial}{\partial r} \left( f^{\sharp}(r e^{it})^2 \right)= \frac{2}{r} f^{\sharp}(r e^{it})^2 \left(h_f(r e^{it}) -1 \right)$$
  we find by a straightforward computation
  \begin{eqnarray*}
   \frac{\partial }{\partial r} \left(  \frac{\pi}{2r} (1-r^2)^2 \Phi_s^{\prime}(r) \right) &=& 
\frac{\partial }{\partial r}   \int\limits_0^{2 \pi} \left(h_f(r e^{it}) -  \left(1-r^4\right) f^{\sharp}(r e^{it})^2  \right) \, dt\\
 &=& 2 \frac{\left(1-r^2\right)\left(1+r^2\right)}{r}\int\limits_0^{2\pi} f^{\sharp}(re^{it})^2 \left[ \frac{1-r^2}{1+r^2} -  h_f(r e^{it}) \right]  \, dt. 
\end{eqnarray*}
Since $f$ is not a spherical isometry, we deduce from (\ref{lowerhf}) that $[ \ldots]  < 0$. Therefore,  $$t \mapsto \pi \left(1-r^2\right)^2 \Phi_s'(r)/(2r)$$ is strictly decreasing, and hence
\hide{
$$1-3r^2 \leq 1-r^2 \leq 2\left(1-r^2\right) \Rightarrow \frac{1-3r^2}{1+r^2} \leq 2 h_f(re^{it}), \quad  \forall r \in (0,1),$$ which leads to 
$$  \int\limits_0^{2 \pi} h_f(r e^{it}) dt -  (1-r^4) \int\limits_0^{2 \pi} f^{\sharp }(r e^{it})^2 dt $$ being decreasing. As a result,  }
\begin{eqnarray*}
\frac{\pi}{2r} \left(1-r^2\right)^2 \Phi_s'(r) &> & \lim \limits_{r \to 1^{+}} \frac{\pi}{2r} \left(1-r^2\right)^2 \Phi_s'(r) \\
&=&  \lim_{r \to 1^{-}}   \int\limits_0^{2 \pi} h_f(r e^{it}) dt -  \left(1-r^4\right) \int\limits_0^{2 \pi} f^{\sharp}(r e^{it})^2 \, dt \\
    & \ge & \lim_{r \to 1^{-}}   \int\limits_0^{2 \pi} h_f(r e^{it}) \, dt  \geq 0 \, , \qquad 0<r<1 \, .
\end{eqnarray*}
Thus $\Phi_s^{\prime}(r)> 0$, so  $ \Phi_s$ is a strictly increasing function of $r$.
The proof of Theorem \ref{thm:monotcurv} is complete.
\end{proof}

\section{Examples}\label{examples}

In this final section we illustrate the monotonic behavior of the functions $\LenS$, $\AreS$ and $\Phi_s$ for various exemplary univalent meromorphic functions that are spherically convex and centrally normalized, spherically convex but not centrally normalized, and not spherically convex. The computations have been carried out with use of \textsc{Mathematica} software. 

A classical example of a spherically convex function that is also centrally normalized, see \cite{mejiapomsph}, is $$f_1(z) = \frac{\sqrt{1+z} -\sqrt{1-z}  }{\sqrt{1+z} +\sqrt{1-z} }, \quad z \in \D.  $$
Its spherical derivative is 
$$f_1^{\sharp}(z) = \frac{1}{|1-z^2|} \cdot \frac{|1 -\sqrt{1-z^2}|}{|z|^2+ |1-\sqrt{1-z^2}|^2}, \quad z \in \D.$$

\begin{figure}[h]
\centering
\begin{minipage}{0.5\textwidth}
  \centering
  \includegraphics[scale=0.55]{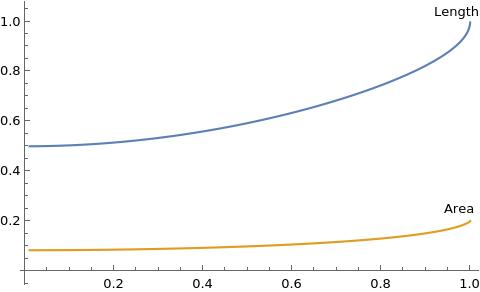}
  \caption{Graphs of $\LenS$ and $\AreS$ for $f_1(z)$}
  \label{fig:1}
\end{minipage}%
\begin{minipage}{0.5\textwidth}
  \centering
  \includegraphics[scale=0.55]{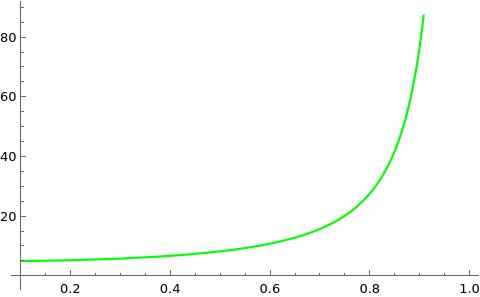}
  \caption{Graph of $\Phi_s$ for $f_1(z)$}
  \label{fig:2}
\end{minipage}
\end{figure}

In Figures \ref{fig:1} and \ref{fig:2}, one can see  that the functions $\LenS$, $\AreS$ and $\Phi_s$ are increasing functions of $r \in (0,1)$, for $f_1$.

In the case where central normalization is omitted, the monotonicity of $\Phi_s$ is disrupted. More specifically, let us define the function $$f_2(z)= e^z, \quad z \in \D.$$
Then $$h_{f_2}(z) = 1+ \frac{1-e^{2 \RE z}}{1+e^{2 \RE z }} \RE z $$
and hence
$$ h_{f_2}(r e^{it})  = 1+ \frac{1-e^{2 r\cos t}}{1+e^{2 r \cos t }} r \cos t=1 -  r \cos t \tanh( r \cos t)  >0,  $$ 
for all $t \in [0,2\pi]$ and $r \in (0,1)$, as we can in Figure \ref{fig:3}. Hence $f_2$ is spherically convex, but $\Phi_s$ is not increasing; see Figure \ref{fig:4}. 

\begin{figure}[h]
\centering
\begin{minipage}{.5\textwidth}
  \centering
  \includegraphics[scale=0.48]{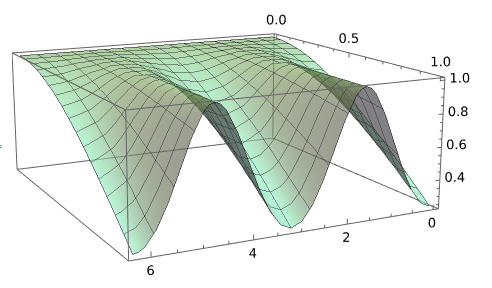}
  \caption{Graph of $h_{f_2}(r e^{it}) $}
  \label{fig:3}
\end{minipage}%
\begin{minipage}{.5\textwidth}
  \centering
  \includegraphics[scale=0.55]{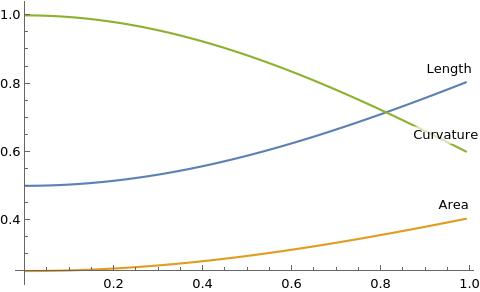}
  \caption{Graphs of $\LenS , \AreS$ and $\Phi_s$ for $e^z$}
  \label{fig:4}
\end{minipage}
\end{figure}

Let us define the function $f_3(z)=z^2 e^z$, for $z \in \D$. This is not a a spherically convex function, since $$h_{f_3}(z) = 1+ \RE \{z+2\} \left(1-\frac{2}{|z+2|^2} - 2 \frac{|z|^4 e^{2 \RE z}}{1+|z|^4 e^{2 \RE z}} \right) $$
and as we see in Figures \ref{fig:5} and \ref{fig:6}, for $r=|z|\geq 0.8$, it attains negative values. 
\begin{figure}[h]
\centering
\begin{minipage}{.5\textwidth}
  \centering
  \includegraphics[scale=0.45]{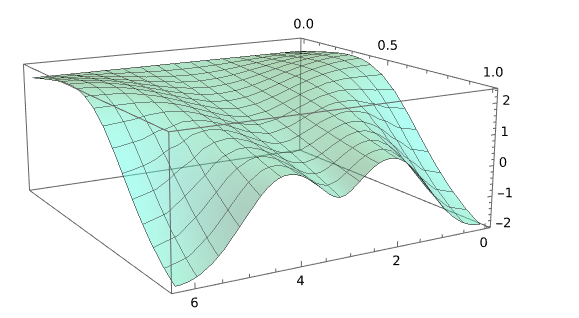}
  \caption{Graph of $h_{f_3}(r e^{it}) $}
  \label{fig:5}
\end{minipage}%
\begin{minipage}{.5\textwidth}
  \centering
  \includegraphics[scale=0.5]{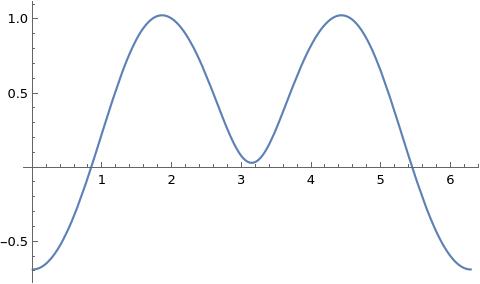}
  \caption{$h_{f_3}(re^{it})$, for $r=0.8$}
  \label{fig:6}
\end{minipage}
\end{figure}

\begin{figure}[h]
\centering
\begin{minipage}{.5\textwidth}
  \centering
  \includegraphics[scale=0.55]{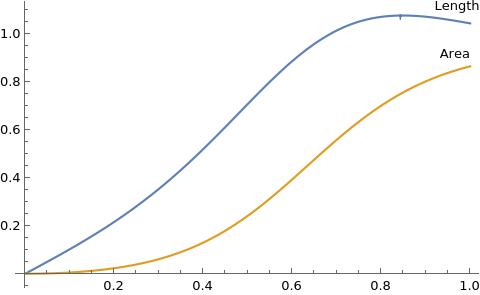}
  \caption{Graphs of $\LenS , \AreS$ for $z^2e^z$}
  \label{fig:7}
\end{minipage}%
\begin{minipage}{.5\textwidth}
  \centering
  \includegraphics[scale=0.55]{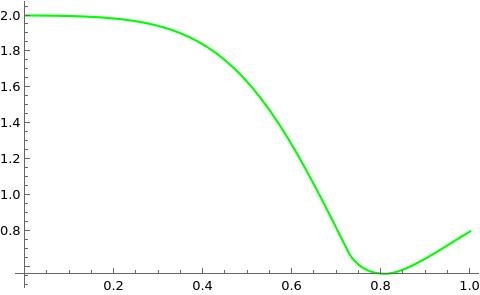}
  \caption{Graph of $\Phi_s$ for $z^2e^z$}
  \label{fig:8}
\end{minipage}
\end{figure}

Calculating its spherical derivative and producing the graphs of $\LenS$, $\AreS$ and $\Phi_s$ for $f_3(z)$, we obtain Figures \ref{fig:7} and \ref{fig:8}. 

In view of those, the significance of spherical convexity in Theorems \ref{thm:monotarea}, \ref{thm:monotlength} and \ref{thm:monotcurv} is straightforward. It is a necessary property that a function is spherically convex, so that the functions $\LenS(r) , \AreS(r)$ and $\Phi_s(r)$ are increasing functions of $r \in (0,1)$.

\section{Conflicts of interests/Competing interests}

The authors declare that there is no conflict of interest.

\end{document}